\documentclass[12pt]{article}
\usepackage{amssymb,amsfonts, amsmath,amsthm,epsfig,tabularx}
\usepackage[title]{appendix}
\usepackage[dvipsnames]{xcolor}

\newtheorem{theorem}{Theorem}

\newtheorem{lemma}[theorem]{Lemma}

\newtheorem{conjecture}[theorem]{Conjecture}
\newtheorem{question}[theorem]{Question}

\newcommand{\Aut}{{\rm Aut}}

\title{On $12$-regular nut graphs}

\author
{Nino Ba\v si\'c\footnote{FAMNIT  \& IAM, University of Primorska, 6000 Koper, Slovenia, and Institute of Mathematics, Physics 
and Mechanics, 1000 Ljubljana, Slovenia, email: {\tt nino.basic@famnit.upr.si}.}
\and
Martin Knor\footnote{Department of Mathematics, 
  Faculty of Civil Engineering, Slovak University of Technology in Bratislava, 
  Radlinsk{\'e}ho 11, 810 05, Bratislava, Slovakia, email: {\tt knor@math.sk}.} 
\and
Riste~{\v S}krekovski\footnote{Faculty of Information Studies, 8000 Novo mesto, Slovenia, 
  and FMF, University of Ljubljana, 1000 Ljubljana, Slovenia, and FAMNIT, University of Primorska, 6000 Koper, Slovenia,
  email: {\tt skrekovski@gmail.com}.}}

\usepackage{hyperref}
\hypersetup{
    colorlinks=true,
    linkcolor=blue,
    citecolor=magenta,      
    urlcolor=OliveGreen
}

\begin{document}
\maketitle

{\abstract{
A nut graph is a simple graph whose adjacency matrix is singular with
$1$-dimensional kernel such that the corresponding eigenvector has no zero entries.
In 2020, Fowler \emph{et al.}\ characterised for each $d \in \{3,4,\ldots,11\}$  
all values $n$ such that there exists a $d$-regular nut graph of order $n$.
In the present paper, we determine all values $n$ for which a $12$-regular nut
graph of order $n$ exists.
We also present a result by which there are infinitely many circulant nut
graphs of degree $d \equiv 0 \pmod 4$ and no circulant nut graph of degree
$d \equiv 2 \pmod 4$.
}}

\bigskip
{\noindent \small \textbf{Keywords:} Nut graph, adjacency matrix, singular matrix, core graph, Fowler construction, regular graph.

\vspace{0.5\baselineskip}
{\noindent \small \textbf{Math.\ Subj.\ Class.\ (2020):} 05C50, 15A18. 
}

%
%

\section{Introduction}

Let $G$ be a simple graph with the vertex set $V(G) = \{0,1,\dots,n{-}1\}$.
Its adjacency matrix $\mathbb A$ is a symmetric $n\times n$ matrix with entries
$a_{i,j}$, where $0\le i,j\le n{-}1$, such that $a_{i,j} = a_{j,i} =1$ if $\{i,j\}$ is an edge of $G$ and $a_{i,j} = a_{j,i} =0$ 
otherwise. Graph $G$ is a {\em nut graph} if $\mathbb A$ has eigenvalue $0$, the
eigenspace corresponding to the eigenvalue $0$ is $1$-dimensional and generated by an eigenvector 
which does not contain a $0$ entry. Observe that if the eigenspace corresponding to $0$ is more than
$1$-dimensional, then there exists an eigenvector containing entry $0$ that is different from
$\bold 0=(0,0,\dots,0)^T$. For an introductory treatment of spectral graph theory, which links graphs
to linear algebra, see e.g.~\cite{BH2012,C1997,CDS1995}.

Nut graphs have been studied in \cite{CFG,FGGPS,GPS,GutmanSciriha-MaxSing,
Sc1998,Sc2007,S,SciCommEigDck09,ScMaxCorSzSing09,SG}, 
see also the webpage \url{https://hog.grinvin.org/Nuts} within the
House of Graphs \cite{HG,nutgen}.
Recently, this concept
was extended to signed graphs \cite{BFPS}.
Nut graphs have chemical applications, see e.g.\ \cite{FGGPS,FPB2020,SF2007}. 
However, in the present paper we consider $12$-regular graphs, so our motivation is purely
mathematical.

In \cite{SG}, Gutman and Sciriha showed that the smallest non-trivial nut graph has order $7$.
In \cite{FPTBS}, Fowler \emph{et al.}\  determined all nut graphs on up to $10$ vertices and
all chemical nut graphs on up to $16$ vertices.
The smallest order for which a regular nut graph exists is $8$; see also
\cite{FGGPS}.
In \cite{FGGPS}, Fowler \emph{et al.}\ presented the following question.

\begin{question}
\label{que:regular}
Is it true that for each $d\ge 3$ there are only finitely many numbers
$n$ such that there does not exist a $d$-regular nut graph of order $n$?
\end{question}

In the attempt to answer Question~{\ref{que:regular}}, the `Fowler
Construction' played an important role; see also \cite{GPS}.
This construction implies the following theorem.

\begin{theorem}
\label{thm:construction}
Let $G$ be a nut graph on $n$ vertices and let $u$ be a vertex of $G$ of degree $d$.
Then there exists a nut graph of order $n+2d$ that is obtained from $G$ by adding $2d$ new vertices
and rearranging the edges in a certain way.
In the newly obtained nut graph the degrees of the new vertices are $d$ and the degrees of the
original vertices are not changed.
\end{theorem}

Obviously, if $G$ is a $d$-regular graph of order $n$, then the new graph is
$d$-regular of order $n+2d$.
Hence, to positively answer the Question~{\ref{que:regular}} for specific degree
$d$, it suffices to find $d$-regular graphs for $2d$ consecutive orders.
In \cite{GPS} ($d=3,4$) and \cite{FGGPS} ($5 \le d \le 11$) the authors found
all pairs $(d,n)$, such that $d\le 11$ and there exists a $d$-regular nut
graph of order $n$.
In the present paper, we extend this result to  $d=12$.
We prove the following statement.

\begin{theorem}
\label{thm:main}
There exists a $12$-regular nut graph of order $n$ if and only if $n\ge 16$.
\end{theorem}

To prove the `positive part' of Theorem~{\ref{thm:main}}, it suffices to
find $12$-regular nut graphs of orders $n\in\{16,17,\dots,39\}$.
We present these graphs in the following section.
For odd orders there is not much to say; we did a computer search and thus we 
provide a
list 
of graphs that we found. However, for even orders we can say more.

A graph $G$ is called {\em vertex-transitive} if all vertices are equivalent under the action of the automorphism
group $\Aut(G)$. In other words, for each
pair of vertices $u, v \in V(G)$ there exist an automorphism $\alpha \in \Aut(G)$
such that $\alpha(u) = v$.
In \cite{FGGPS}, the following necessary condition for a
vertex-transitive nut graph was given.

\begin{theorem}
\label{thm:necessary_vt}
Let $G$ be a vertex-transitive nut graph of degree $d$ on $n$ vertices.
Then $n$ and $d$ satisfy the following conditions. Either
\begin{enumerate}
\item
$d\equiv 0\pmod 4$, $n\equiv 0\pmod 2$ and $n\ge d+4$, or
\item
$d\equiv 0\pmod 2$, $n\equiv 0\pmod 4$ and $n\ge d+6$.
\end{enumerate}
\end{theorem}

The existence of vertex-transitive nut graphs is interesting on its own, see
\cite[Question~4]{FGGPS}.
For our research it is important that, by Theorem~{\ref{thm:necessary_vt}},
there may exist vertex-transitive
$12$-regular graphs of even orders $n\ge 16$.
We found such graphs among circulant graphs.

%
%

\section{Results}

We start with the `negative part' of Theorem~{\ref{thm:main}}.
There is only one $12$-regular graph of order $13$, namely the complete
graph $K_{13}$, and it is not a nut graph.
The unique $12$-regular graph of order $14$ is obtained by removing a
matching from $K_{14}$, and again, this graph is not a nut graph.
Finally, there are $17$ graphs of order $15$ which are $12$-regular.
They are obtained by removing a $2$-factor from $K_{15}$. Using the
SageMath software \cite{SageMath} we analysed all such graphs and concluded 
that none of them is a nut graph.

Now we turn our attention to the `positive part' of Theorem~{\ref{thm:main}}.
We start with more general results for even orders.
The following lemma is in fact hidden in the text preceding Proposition~1 in
\cite{GPS}.
We decided to present it here in a slightly more general frame together with
its short proof.

\begin{lemma}
\label{lem:sum}
Let $G$ be a $d$-regular graph on $n$ vertices such that its adjacency
matrix $\mathbb A$ is singular. 
Then for every eigenvector $\bold c=(c_0,c_1,\dots,c_{n-1})^T$ corresponding
to eigenvalue $0$ we have 
\begin{equation*}
\sum_{i=0}^{n-1}c_i=0.
\end{equation*}
\end{lemma}

\begin{proof}
Let $\mathbb A=(\bold a_0,\bold a_1,\dots,\bold a_{n-1})^T$.
Then
$\mathbb A\bold c=(\bold a_0\bold c,\bold a_1\bold c,\dots,\bold a_{n-1}\bold c)^T= \bold 0$,
and thus $\sum_{i=0}^{n-1}\bold a_i\bold c=0$.
However, $\sum_{i=0}^{n-1}\bold a_i\bold c=\sum_{i=0}^{n-1}dc_i$,
and since $d>0$, we have $\sum_{i=0}^{n-1}c_i=0$.
\end{proof}

Let $V=\{0,1,\dots,n{-}1\}$ and let $1\le a_1<a_2<\dots<a_t\le\frac n2$.
By $C(n,\{a_1,a_2,\dots,a_t\})$ we denote a graph on the vertex set $V$ in which
two vertices $i,j\in V$ are adjacent if and only if $|i-j|=a_k$,  
where $1\le k\le t$.
The graph $C(n,\{a_1,a_2,\dots,a_t\})$ is called a {\em circulant graph} and
it is regular.
Its degree is $2t-1$ if $a_t=\frac n2$ and $2t$ otherwise.
In fact, circulant graphs are vertex-transitive since $\varphi\colon i\to i+1$
is an automorphism of $C(n,\{a_1,a_2,\dots,a_t\})$ (the addition is modulo
$n$).

Circulant graphs are easy to describe and easy to handle.
Therefore, it would be nice if there were many nut graphs among them.
We prove one positive and one negative result about circulant graphs.
We start with the following lemma.

\begin{lemma}
\label{lem:eigenvector}
Let $G=C(n,\{a_1,a_2,\dots,a_t\})$ be a circulant nut graph, and let
$\mathbb A$ be its adjacency matrix.
Then $(1,-1,1,-1,\dots)$ is an eigenvector corresponding to eigenvalue $0$.
\end{lemma}

\begin{proof}
We use the well-known fact that if $\bold b$ and $\bold c$ are eigenvectors
corresponding to eigenvalue $\lambda$, then $\bold b+\bold c$ is also an
eigenvector corresponding to eigenvalue $\lambda$.

Let $\bold b=(b_0,b_1,\dots,b_{n-1})^T$ be an eigenvector corresponding to
$0$.
Denote $b_0=p$ and $b_1=q$.
Since $\varphi\colon i\to 2-i$ is an automorphism of $G$ (the addition being
modulo $n$), there is an eigenvector $\bold c=(c_0,c_1,\dots,c_{n-1})^T$
such that $c_{2-i}=-b_i$, $0\le i\le n-1$.
Then $c_1=-b_1=-q$ and $c_2=-b_0=-p$.
Since $b_1+c_1=0$ and $\bold b+\bold c$ is an eigenvector, we must have
$\bold b+\bold c=\bold 0$ because $G$ is a nut graph.
Hence, $b_2+c_2=0$ and therefore $b_2=p$.
Now repeating the process we get $\bold b=(p,q,p,q,\dots)$.
Observe that $n$ is even by Theorem~{\ref{thm:necessary_vt}}.
Thus, by Lemma~{\ref{lem:sum}}, we have $q=-p$ and so $(1,-1,1,-1,\dots)$ is
an eigenvector corresponding to eigenvalue $0$.
\end{proof}

Our negative result covers all circulant graphs of degree
$d\equiv 2\pmod 4$.

\begin{theorem}
\label{thm:d=2m4}
There is no circulant nut graph of degree $d$ if $d\equiv 2\pmod 4$.
\end{theorem}

\begin{proof}
Let $d\equiv 2\pmod 4$.
Denote $t=\frac d2$.
Observe that $t$ is an odd number.
By way of contradiction, assume that $G=C(n,\{a_1,a_2,\dots,a_t\})$ is a
circulant nut graph.
Then $n$ is even by Theorem~{\ref{thm:necessary_vt}}.
Let $\mathbb A=(\bold a_0,\bold a_1,\dots,\bold a_{n-1})^T$ be the adjacency
matrix of $G$.
By Lemma~{\ref{lem:eigenvector}}, $\bold c=(1,-1,1,-1,\dots)^T$ is an
eigenvector corresponding to eigenvalue $0$, so that
$\mathbb A\bold c=\bold 0$, and in particular $\bold a_0\bold c=0$.
However,
\begin{equation*}
\bold a_0\bold c=c_{a_1}+c_{a_2}+\dots+c_{a_t}+c_{n-a_1}+c_{n-a_2}+\dots+c_{n-a_t}.
\end{equation*}
Since $c_{a_i}=c_{n-a_i}$ for every $i$, $1\le i\le t$ (observe that the
difference between indices $a_i$ and $n-a_i$ is even), we have $\bold
a_0\bold c=2(c_{a_1}+c_{a_2}+\dots+c_{a_t})$, which implies that
$c_{a_1}+c_{a_2}+\dots+c_{a_t}=0$.
However, sum of odd number of odd numbers cannot be an even number, a
contradiction.
\end{proof}

Now we prove the positive result.

\begin{theorem}
\label{thm:d=0m4}
Let $d\equiv 0\pmod 4$ and let $n$ be even.
Then $C(n,\{1,2,\dots,\frac {d}{2}\})$ is a nut graph if and only if $\frac {d}{2}+1$
is coprime to $n$ and $\frac {d}{4}$ is coprime to $\frac n2$.
\end{theorem}

\begin{proof}
Let $t=\frac {d}{2}$.
Then $t$ is even and the graph is $G=C(n,\{1,2,\dots,t\})$.

Let $\mathbb A$ be the adjacency matrix of $G$.
By Lemma~{\ref{lem:eigenvector}}, $\bold b=(1,-1,1,-1,\dots)^T$ is an eigenvector
of $\mathbb A$ corresponding to eigenvalue $0$.
Thus $\mathbb A\bold b=\bold 0$.
Our aim is to show that if $t+1$ is coprime to $n$ and $\frac t2$
is coprime to $\frac n2$, then
$\mathbb A\bold c=\bold 0$ if and only if $\bold c$ is a multiple of
$\bold b$.

So let $\mathbb A\bold c=\bold 0$, where $\bold c=(c_0,c_1,\dots,c_{n-1})^T$.
Let $\mathbb A=(\bold a_0,\bold a_1,\dots,\bold a_{n-1})^T$. Then
\begin{align*}
\bold a_t\bold c &= c_0+c_1+\dots+c_{t-1}+c_{t+1}+c_{t+2}+\dots+c_{2t}=0,\\
\bold a_{t+1}\bold c &= c_1+c_2+\dots+c_t+c_{t+2}+c_{t+3}+\dots+c_{2t+1}=0.
\end{align*}
Subtracting the two equations we get
\begin{equation*}
\bold a_t\bold c-\bold a_{t+1}\bold c = c_0-c_t+c_{t+1}-c_{2t+1}=0,
\end{equation*}
and analogously
\begin{equation*}
\bold a_{2t+1}\bold c-\bold a_{2t+2}\bold c = c_{t+1}-c_{2t+1}+c_{2t+2}-c_{3t+2}=0.
\end{equation*}
This gives
\begin{equation*}
c_0-c_t=c_{2t+2}-c_{3t+2},
\end{equation*}
and analogously
\begin{align*}
c_{2t+2}-c_{3t+2} & =c_{4t+4}-c_{5t+4}, \\
c_{4t+4}-c_{5t+4} & =c_{6t+6}-c_{7t+5}, \qquad \text{etc.}
\end{align*}
So if the odd number $t+1$ is coprime to even number $n$,
we get
\begin{equation*}
c_0-c_t = c_{2(t+1)}-c_{t+2(t+1)}=
\cdots = c_2-c_{t+2},
\end{equation*}
which gives
\begin{equation*}
c_2-c_0=c_{t+2}-c_t,
\end{equation*}
and analogously we get
\begin{align*}
c_{t+2}-c_t & =c_{2t+2}-c_{2t}, \\
 c_{2t+2}-c_{2t} & =c_{3t+2}-c_{3t}, \qquad \text{etc.}
\end{align*}
Here, $t$ and $n$ are both even.
But if $\frac t2$ is coprime to $\frac n2$ then
\begin{equation*}
c_2-c_0=c_{t+2}-c_t=\dots=c_4-c_2.
\end{equation*}
Hence,
\begin{equation*}
c_2-c_0=c_4-c_2=c_6-c_4=\cdots
\end{equation*}
Now, if $c_2>c_0$ then $c_0<c_2<c_4<\dots <c_0$, a contradiction.
Analogously, if $c_2<c_0$ then $c_0>c_2>c_4>\dots >c_0$, a contradiction.
So $c_0=c_2=\dots=c_{n-2}$ and analogously $c_1=c_3=\dots=c_{n-1}$.
Hence if $c_0=p$, then $\bold c=(p,-p,p,-p,\dots)$ by Lemma~{\ref{lem:sum}},
and the eigenspace corresponding to eigenvalue $0$ is $1$-dimensional.

Now suppose that $t+1$ is not coprime to $n$.
Set $\bold b=\bold 0$.
We will change some entries of $\bold b$.
Since $t+1$ is odd, there is an even $k$ such that $(t+1)k\equiv 0\pmod n$
and $1\le k<n$.
Set
\begin{equation*}
b_0=1,\quad b_{t+1}=-1,\quad b_{2(t+1)}=1, \quad b_{3(t+1)}=-1,\quad \ldots,
\end{equation*}
where the indices are modulo $n$.
We have changed $k$ entries of $\bold b$ and since $k$ is even, the last
changed entry has value $-1$.
Thus some entries of $\bold b$ remained $0$'s and nevertheless
$\mathbb A\bold b=\bold 0$, since if $j$-th entry of $\bold a_i$ is
$1$, then either $(j+(t+1))$-th or $(j-(t+1))$-th (modulo $n$) entry
of $\bold a_i$ is also $1$ (while the other is $0$).
Hence, $G$ is not a nut graph in this case.

Finally, suppose that $\frac t2$ is not coprime to $\frac n2$.
Then there exist a number $k$ such that $k\mid\frac t2$, $k\mid\frac n2$ and $k>1$.
Again, set $\bold b=\bold 0$.
We will change some entries of $\bold b$.
Set
\begin{equation*}
b_0=b_2=b_4=\dots =b_{2(k-2)}=1 \text{\qquad and\qquad}
b_{2(k-1)}  =-(k-1),
\end{equation*}
and repeat this pattern for all even indices of $\bold b$.
Since $k\mid\frac n2$, this pattern is repeated exactly $\frac{n}{2k}$ times.
And since every $\bold a_i$ contains two disjoint sets of $t$ consecutive
$1$'s, we have $\mathbb A\bold b=\bold 0$. But half of the entries of
$\bold b$ are $0$'s and therefore $G$ is not a nut graph.
\end{proof}

Observe that the only requirement for $n$ in Theorem~{\ref{thm:d=0m4}} is
that $n$ is even and $n>d$.
However, if $n=d+2$ then $\frac d2+1$ is not coprime to $n$, and so $n\ge
d+4$.
Hence, by Theorem~{\ref{thm:d=0m4}}, for $d=12$ the following circulant
graphs are nut graphs: 
\begin{align*}
 & C(16,\{1,2,3,4,5,6\}), & & C(20,\{1,2,3,4,5,6\}), & & C(22,\{1,2,3,4,5,6\}), \\
 & C(26,\{1,2,3,4,5,6\}), & & C(32,\{1,2,3,4,5,6\}), & & C(34,\{1,2,3,4,5,6\}), \text{ and} \\
 & C(38,\{1,2,3,4,5,6\}).
\end{align*}
Using computer \cite{SageMath} we found that nut graphs are also the following graphs:
\begin{align*}
& C(18,\{1,2,3,4,5,8\}), & & C(24,\{1,2,3,4,5,8\}), & & C(28,\{1,2,3,4,5,10\}), \\
& C(30,\{1,2,3,4,5,8\}), \text{ and} & & C(36,\{1,2,3,4,5,8\}).
\end{align*}
We pose the following conjecture.

\begin{conjecture}
\label{conj:12}
For every even $n$, $n\ge 16$, there exist a circulant nut graph
$C(n,\{a_1,a_2,\dots,a_6\})$ of degree $12$.
\end{conjecture}

\noindent We also give a more general conjecture.

\begin{conjecture}
\label{conj:4|d}
For every $d$, where $d\equiv 0\pmod 4$, and for every even $n$, $n\ge d+4$,
there exists a circulant nut graph $C(n,\{a_1,a_2,\dots,a_{d/2}\})$ of degree $d$.
\end{conjecture}

By Theorem~{\ref{thm:necessary_vt}}, if $n$ is odd then there is no vertex-transitive nut
graph of order $n$ and degree $12$. In this case all graphs were found by a computer search.
If $G$ is a regular graph that contains edges $u_1v_1$ and $u_2v_2$ but does
not contain edges $u_1v_2,u_2v_1$, then {\em rewiring}
(i.e.\ removing edges $u_1v_1,u_2v_2$ and adding edges $u_1v_2,u_2v_1$) 
yields another regular graph.
Our approach was to start with a ``nice" $12$-regular graph of odd order
and perferm a sequence of rewirings.
In this way all graphs in the Appendix were obtained.
For instance, the graph on $21$ vertices, whose eigenvector contains 
only values $1$ and $-2$, was obtained from $C(21,\{1,2,3,4,5,6\})$ by
removing the edges $(0,16)$ and $(2,7)$ and adding the edges $(0,7)$ and
$(2,16)$.
For $n=13$ this method seems to be too time-consuming.

\vskip 1pc
\noindent{\bf Acknowledgements.}~The work of the first author is supported in part by the Slovenian Research Agency
(research program P1-0294 and research projects J1-9187, J1-1691, N1-0140 and J1-2481). The second author acknowledges
partial support by Slovak research grants APVV-15-0220, APVV-17-0428,
VEGA 1/0142/17 and VEGA 1/0238/19. The research of the third author was partially supported by the Slovenian Research Agency (ARRS),
research program P1-0383 and research project J1-1692.

\vskip 1pc
\noindent{\bf ORCID iD}

\noindent Nino Bašić \href{https://orcid.org/0000-0002-6555-8668}{\includegraphics[scale=0.05]{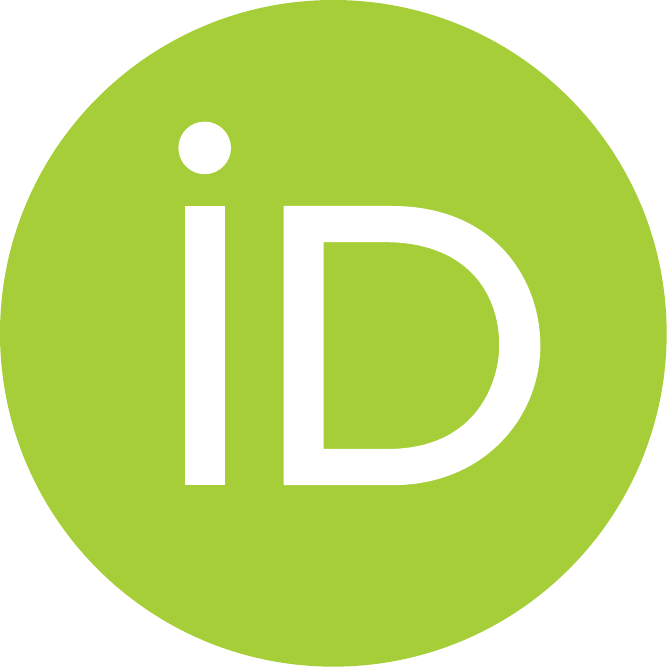} https://orcid.org/0000-0002-6555-8668}

\noindent Martin Knor \href{https://orcid.org/0000-0003-3555-3994}{\includegraphics[scale=0.05]{ORCID_icon.pdf} https://orcid.org/0000-0003-3555-3994}

\noindent Riste Škrekovski \href{https://orcid.org/0000-0001-6851-3214}{\includegraphics[scale=0.05]{ORCID_icon.pdf} https://orcid.org/0000-0001-6851-3214}



\newpage
\begin{appendices}

\section{$\boldsymbol{12}$-regular nut graphs of odd orders}

Here, we list one $12$-regular nut graph of odd order $n$ for each
$n\in \{17,19,\ldots,39\}$.
Each graph is given in the adjacency-lists (of neighbours of each vertex) representaion,
formatted as a Python dictionary.
We also give the corresponding kernel eigenvector $\bold c$ as a list of integer entries.

\paragraph{Order $\boldsymbol{n = 17}$.} ~\\
{\small \{0: [1, 2, 3, 4, 5, 8, 9, 10, 11, 12, 15, 16], 1: [0, 2, 3, 4, 6, 7, 8, 9, 10, 11, 15, 16], 2: [0, 1, 4, 5, 6, 7, 8, 9, 10, 11, 13, 15], 3: [0, 1, 4, 6, 7, 8, 9, 11, 12, 14, 15, 16], 4: [0, 1, 2, 3, 5, 6, 8, 9, 10, 11, 13, 16], 5: [0, 2, 4, 6, 7, 8, 9, 10, 12, 13, 14, 15], 6: [1, 2, 3, 4, 5, 7, 8, 9, 12, 13, 14, 15], 7: [1, 2, 3, 5, 6, 8, 10, 11, 12, 13, 14, 16], 8: [0, 1, 2, 3, 4, 5, 6, 7, 10, 11, 13, 14], 9: [0, 1, 2, 3, 4, 5, 6, 10, 12, 13, 14, 16], 10: [0, 1, 2, 4, 5, 7, 8, 9, 12, 14, 15, 16], 11: [0, 1, 2, 3, 4, 7, 8, 12, 13, 14, 15, 16], 12: [0, 3, 5, 6, 7, 9, 10, 11, 13, 14, 15, 16], 13: [2, 4, 5, 6, 7, 8, 9, 11, 12, 14, 15, 16], 14: [3, 5, 6, 7, 8, 9, 10, 11, 12, 13, 15, 16], 15: [0, 1, 2, 3, 5, 6, 10, 11, 12, 13, 14, 16], 16: [0, 1, 3, 4, 7, 9, 10, 11, 12, 13, 14, 15]\} }

\vspace{0.5\baselineskip}
\noindent $\bold c = {}$[$3$, $-3$, $-2$, $2$, $1$, $2$, $-1$, $-2$, $3$, $-1$, $-1$, $1$, $1$, $-1$, $1$, $-1$, $-2$]

\paragraph{Order $\boldsymbol{n = 19}$.} ~\\
{\small \{0: [1, 2, 5, 7, 9, 10, 11, 12, 13, 14, 16, 18], 1: [0, 3, 5, 6, 7, 10, 12, 13, 14, 15, 17, 18], 2: [0, 4, 6, 7, 8, 9, 10, 11, 12, 16, 17, 18], 3: [1, 6, 7, 8, 10, 11, 12, 13, 14, 16, 17, 18], 4: [2, 5, 6, 7, 8, 11, 12, 13, 14, 15, 17, 18], 5: [0, 1, 4, 7, 8, 9, 11, 12, 13, 14, 15, 17], 6: [1, 2, 3, 4, 7, 8, 9, 10, 14, 15, 16, 17], 7: [0, 1, 2, 3, 4, 5, 6, 8, 9, 11, 15, 16], 8: [2, 3, 4, 5, 6, 7, 9, 11, 14, 15, 17, 18], 9: [0, 2, 5, 6, 7, 8, 10, 11, 12, 13, 16, 17], 10: [0, 1, 2, 3, 6, 9, 11, 12, 13, 14, 16, 18], 11: [0, 2, 3, 4, 5, 7, 8, 9, 10, 16, 17, 18], 12: [0, 1, 2, 3, 4, 5, 9, 10, 13, 14, 15, 16], 13: [0, 1, 3, 4, 5, 9, 10, 12, 14, 15, 16, 17], 14: [0, 1, 3, 4, 5, 6, 8, 10, 12, 13, 15, 18], 15: [1, 4, 5, 6, 7, 8, 12, 13, 14, 16, 17, 18], 16: [0, 2, 3, 6, 7, 9, 10, 11, 12, 13, 15, 18], 17: [1, 2, 3, 4, 5, 6, 8, 9, 11, 13, 15, 18], 18: [0, 1, 2, 3, 4, 8, 10, 11, 14, 15, 16, 17]\} }

\vspace{0.5\baselineskip}
\noindent $\bold c = {}$[$5$, $10$, $6$, $-10$, $-3$, $-1$, $4$, $-1$, $-5$, $1$, $1$, $-5$, $-4$, $-3$, $-4$, $2$, $-4$, $7$, $4$]

\paragraph{Order $\boldsymbol{n = 21}$.} ~\\
{\small \{0: [1, 2, 3, 4, 5, 6, 7, 15, 17, 18, 19, 20], 1: [0, 2, 3, 4, 5, 6, 7, 16, 17, 18, 19, 20], 2: [0, 1, 3, 4, 5, 6, 8, 16, 17, 18, 19, 20], 3: [0, 1, 2, 4, 5, 6, 7, 8, 9, 18, 19, 20], 4: [0, 1, 2, 3, 5, 6, 7, 8, 9, 10, 19, 20], 5: [0, 1, 2, 3, 4, 6, 7, 8, 9, 10, 11, 20], 6: [0, 1, 2, 3, 4, 5, 7, 8, 9, 10, 11, 12], 7: [0, 1, 3, 4, 5, 6, 8, 9, 10, 11, 12, 13], 8: [2, 3, 4, 5, 6, 7, 9, 10, 11, 12, 13, 14], 9: [3, 4, 5, 6, 7, 8, 10, 11, 12, 13, 14, 15], 10: [4, 5, 6, 7, 8, 9, 11, 12, 13, 14, 15, 16], 11: [5, 6, 7, 8, 9, 10, 12, 13, 14, 15, 16, 17], 12: [6, 7, 8, 9, 10, 11, 13, 14, 15, 16, 17, 18], 13: [7, 8, 9, 10, 11, 12, 14, 15, 16, 17, 18, 19], 14: [8, 9, 10, 11, 12, 13, 15, 16, 17, 18, 19, 20], 15: [0, 9, 10, 11, 12, 13, 14, 16, 17, 18, 19, 20], 16: [1, 2, 10, 11, 12, 13, 14, 15, 17, 18, 19, 20], 17: [0, 1, 2, 11, 12, 13, 14, 15, 16, 18, 19, 20], 18: [0, 1, 2, 3, 12, 13, 14, 15, 16, 17, 19, 20], 19: [0, 1, 2, 3, 4, 13, 14, 15, 16, 17, 18, 20], 20: [0, 1, 2, 3, 4, 5, 14, 15, 16, 17, 18, 19]\} }

\vspace{0.5\baselineskip}
\noindent $\bold c = {}$[$1$, $-2$, $1$, $1$, $-2$, $1$, $1$, $-2$, $1$, $1$, $-2$, $1$, $1$, $-2$, $1$, $1$, $-2$, $1$, $1$, $-2$, $1$]

\paragraph{Order $\boldsymbol{n = 23}$.} ~\\
{\small \{0: [1, 2, 4, 6, 7, 8, 10, 11, 13, 19, 20, 21], 1: [0, 4, 5, 6, 7, 9, 11, 13, 16, 17, 20, 22], 2: [0, 3, 4, 6, 8, 11, 12, 13, 16, 19, 20, 21], 3: [2, 4, 5, 8, 9, 10, 12, 13, 14, 16, 17, 18], 4: [0, 1, 2, 3, 6, 7, 8, 14, 15, 16, 21, 22], 5: [1, 3, 7, 10, 11, 12, 14, 15, 17, 18, 19, 20], 6: [0, 1, 2, 4, 11, 12, 14, 17, 18, 19, 20, 22], 7: [0, 1, 4, 5, 10, 11, 12, 16, 18, 19, 21, 22], 8: [0, 2, 3, 4, 9, 10, 12, 13, 15, 16, 21, 22], 9: [1, 3, 8, 10, 11, 13, 14, 15, 18, 19, 21, 22], 10: [0, 3, 5, 7, 8, 9, 11, 13, 17, 18, 19, 21], 11: [0, 1, 2, 5, 6, 7, 9, 10, 13, 14, 15, 20], 12: [2, 3, 5, 6, 7, 8, 13, 14, 17, 19, 20, 22], 13: [0, 1, 2, 3, 8, 9, 10, 11, 12, 14, 15, 19], 14: [3, 4, 5, 6, 9, 11, 12, 13, 15, 16, 17, 20], 15: [4, 5, 8, 9, 11, 13, 14, 17, 18, 20, 21, 22], 16: [1, 2, 3, 4, 7, 8, 14, 17, 18, 20, 21, 22], 17: [1, 3, 5, 6, 10, 12, 14, 15, 16, 19, 20, 21], 18: [3, 5, 6, 7, 9, 10, 15, 16, 19, 20, 21, 22], 19: [0, 2, 5, 6, 7, 9, 10, 12, 13, 17, 18, 22], 20: [0, 1, 2, 5, 6, 11, 12, 14, 15, 16, 17, 18], 21: [0, 2, 4, 7, 8, 9, 10, 15, 16, 17, 18, 22], 22: [1, 4, 6, 7, 8, 9, 12, 15, 16, 18, 19, 21]\} }

\vspace{0.5\baselineskip}
\noindent $\bold c = {}$[$6$, $-24$, $-7$, $13$, $39$, $1$, $27$, $4$, $-18$, $-4$, $10$, $3$, $-14$, $-14$, $28$, $1$, $-22$, $-2$, $3$, $6$, $-28$, $2$, $-10$]

\paragraph{Order $\boldsymbol{n = 25}$.} ~\\
{\small \{0: [3, 4, 5, 7, 9, 10, 12, 13, 17, 19, 22, 23], 1: [2, 3, 5, 11, 12, 15, 16, 18, 19, 20, 21, 23], 2: [1, 3, 4, 5, 10, 13, 14, 17, 20, 21, 23, 24], 3: [0, 1, 2, 5, 8, 10, 14, 16, 20, 21, 23, 24], 4: [0, 2, 6, 8, 9, 10, 11, 13, 18, 21, 23, 24], 5: [0, 1, 2, 3, 10, 13, 14, 17, 18, 19, 20, 24], 6: [4, 8, 9, 10, 11, 12, 14, 17, 19, 20, 21, 22], 7: [0, 8, 9, 11, 12, 15, 16, 18, 19, 22, 23, 24], 8: [3, 4, 6, 7, 9, 10, 11, 13, 17, 18, 22, 23], 9: [0, 4, 6, 7, 8, 10, 11, 12, 14, 15, 18, 21], 10: [0, 2, 3, 4, 5, 6, 8, 9, 15, 16, 17, 18], 11: [1, 4, 6, 7, 8, 9, 12, 13, 14, 17, 19, 20], 12: [0, 1, 6, 7, 9, 11, 13, 14, 15, 18, 21, 22], 13: [0, 2, 4, 5, 8, 11, 12, 16, 20, 21, 22, 23], 14: [2, 3, 5, 6, 9, 11, 12, 15, 16, 17, 19, 22], 15: [1, 7, 9, 10, 12, 14, 16, 17, 19, 20, 22, 24], 16: [1, 3, 7, 10, 13, 14, 15, 17, 18, 19, 20, 24], 17: [0, 2, 5, 6, 8, 10, 11, 14, 15, 16, 21, 23], 18: [1, 4, 5, 7, 8, 9, 10, 12, 16, 21, 22, 24], 19: [0, 1, 5, 6, 7, 11, 14, 15, 16, 21, 22, 24], 20: [1, 2, 3, 5, 6, 11, 13, 15, 16, 22, 23, 24], 21: [1, 2, 3, 4, 6, 9, 12, 13, 17, 18, 19, 23], 22: [0, 6, 7, 8, 12, 13, 14, 15, 18, 19, 20, 24], 23: [0, 1, 2, 3, 4, 7, 8, 13, 17, 20, 21, 24], 24: [2, 3, 4, 5, 7, 15, 16, 18, 19, 20, 22, 23]\} }

\vspace{0.5\baselineskip}
\noindent $\bold c = {}$[$29$, $20$, $-31$, $7$, $5$, $-13$, $32$, $-19$, $-12$, $1$, $31$, $-12$, $-8$, $-6$, $-49$, $17$, $3$, $-17$, $-21$, $20$, $33$, $7$, $1$, $-2$, $-16$]

\paragraph{Order $\boldsymbol{n = 27}$.} ~\\
{\small \{0: [2, 3, 4, 5, 6, 7, 21, 22, 23, 24, 25, 26], 1: [2, 3, 4, 5, 6, 7, 8, 22, 23, 24, 25, 26], 2: [0, 1, 3, 4, 5, 6, 7, 8, 23, 24, 25, 26], 3: [0, 1, 2, 4, 5, 6, 7, 8, 9, 24, 25, 26], 4: [0, 1, 2, 3, 5, 6, 7, 8, 9, 10, 25, 26], 5: [0, 1, 2, 3, 4, 6, 7, 8, 9, 10, 11, 26], 6: [0, 1, 2, 3, 4, 5, 7, 8, 9, 10, 11, 12], 7: [0, 1, 2, 3, 4, 5, 6, 9, 10, 11, 12, 13], 8: [1, 2, 3, 4, 5, 6, 9, 10, 11, 12, 13, 14], 9: [3, 4, 5, 6, 7, 8, 10, 11, 12, 13, 14, 15], 10: [4, 5, 6, 7, 8, 9, 11, 12, 13, 14, 15, 16], 11: [5, 6, 7, 8, 9, 10, 12, 13, 14, 15, 16, 17], 12: [6, 7, 8, 9, 10, 11, 13, 14, 15, 16, 17, 18], 13: [7, 8, 9, 10, 11, 12, 14, 15, 16, 17, 18, 19], 14: [8, 9, 10, 11, 12, 13, 15, 16, 17, 18, 19, 20], 15: [9, 10, 11, 12, 13, 14, 16, 17, 18, 19, 20, 21], 16: [10, 11, 12, 13, 14, 15, 17, 18, 19, 20, 21, 22], 17: [11, 12, 13, 14, 15, 16, 18, 19, 20, 21, 22, 23], 18: [12, 13, 14, 15, 16, 17, 19, 20, 21, 22, 23, 24], 19: [13, 14, 15, 16, 17, 18, 20, 21, 22, 23, 24, 25], 20: [14, 15, 16, 17, 18, 19, 21, 22, 23, 24, 25, 26], 21: [0, 15, 16, 17, 18, 19, 20, 22, 23, 24, 25, 26], 22: [0, 1, 16, 17, 18, 19, 20, 21, 23, 24, 25, 26], 23: [0, 1, 2, 17, 18, 19, 20, 21, 22, 24, 25, 26], 24: [0, 1, 2, 3, 18, 19, 20, 21, 22, 23, 25, 26], 25: [0, 1, 2, 3, 4, 19, 20, 21, 22, 23, 24, 26], 26: [0, 1, 2, 3, 4, 5, 20, 21, 22, 23, 24, 25]\} }

\vspace{0.5\baselineskip}
\noindent $\bold c = {}$[$1$, $-2$, $1$, $1$, $-2$, $1$, $1$, $-2$, $1$, $1$, $-2$, $1$, $1$, $-2$, $1$, $1$, $-2$, $1$, $1$, $-2$, $1$, $1$, $-2$, $1$, $1$, $-2$, $1$]

\paragraph{Order $\boldsymbol{n = 29}$.} ~\\
{\small \{0: [1, 3, 5, 6, 9, 10, 11, 13, 14, 19, 26, 28], 1: [0, 2, 4, 5, 11, 16, 17, 18, 19, 21, 26, 27], 2: [1, 3, 8, 9, 10, 11, 13, 24, 25, 26, 27, 28], 3: [0, 2, 12, 13, 17, 20, 21, 23, 24, 25, 26, 27], 4: [1, 5, 6, 9, 11, 15, 16, 17, 20, 22, 23, 28], 5: [0, 1, 4, 7, 12, 15, 16, 19, 20, 22, 24, 25], 6: [0, 4, 7, 8, 9, 11, 15, 17, 18, 19, 21, 22], 7: [5, 6, 8, 11, 12, 13, 15, 16, 18, 20, 22, 24], 8: [2, 6, 7, 10, 12, 15, 19, 20, 21, 24, 26, 27], 9: [0, 2, 4, 6, 12, 14, 15, 20, 22, 23, 24, 27], 10: [0, 2, 8, 13, 16, 17, 18, 20, 21, 23, 25, 26], 11: [0, 1, 2, 4, 6, 7, 12, 16, 17, 19, 20, 23], 12: [3, 5, 7, 8, 9, 11, 14, 15, 18, 19, 21, 25], 13: [0, 2, 3, 7, 10, 14, 15, 21, 23, 25, 27, 28], 14: [0, 9, 12, 13, 15, 18, 22, 23, 24, 26, 27, 28], 15: [4, 5, 6, 7, 8, 9, 12, 13, 14, 18, 22, 27], 16: [1, 4, 5, 7, 10, 11, 18, 20, 21, 25, 27, 28], 17: [1, 3, 4, 6, 10, 11, 18, 19, 22, 24, 27, 28], 18: [1, 6, 7, 10, 12, 14, 15, 16, 17, 19, 23, 24], 19: [0, 1, 5, 6, 8, 11, 12, 17, 18, 23, 26, 27], 20: [3, 4, 5, 7, 8, 9, 10, 11, 16, 25, 26, 28], 21: [1, 3, 6, 8, 10, 12, 13, 16, 22, 23, 25, 26], 22: [4, 5, 6, 7, 9, 14, 15, 17, 21, 24, 25, 27], 23: [3, 4, 9, 10, 11, 13, 14, 18, 19, 21, 24, 28], 24: [2, 3, 5, 7, 8, 9, 14, 17, 18, 22, 23, 28], 25: [2, 3, 5, 10, 12, 13, 16, 20, 21, 22, 26, 28], 26: [0, 1, 2, 3, 8, 10, 14, 19, 20, 21, 25, 28], 27: [1, 2, 3, 8, 9, 13, 14, 15, 16, 17, 19, 22], 28: [0, 2, 4, 13, 14, 16, 17, 20, 23, 24, 25, 26]\} }

\vspace{0.5\baselineskip}
\noindent $\bold c = {}$[$1$, $1$, $37$, $-13$, $-20$, $-42$, $21$, $-5$, $-36$, $25$, $5$, $30$, $41$, $-25$, $21$, $-6$, $6$, $17$, $34$, $-34$, $-14$, $-13$, $7$, $-51$, $-16$, $39$, $5$, $-21$, $6$]

\paragraph{Order $\boldsymbol{n = 31}$.} ~\\
{\small \{0: [5, 10, 12, 13, 17, 18, 21, 22, 24, 26, 27, 29], 1: [3, 6, 7, 8, 10, 14, 17, 20, 23, 25, 27, 30], 2: [4, 7, 9, 10, 18, 21, 22, 23, 24, 25, 27, 28], 3: [1, 4, 5, 11, 13, 16, 17, 18, 19, 24, 25, 29], 4: [2, 3, 5, 11, 12, 13, 18, 21, 25, 26, 28, 29], 5: [0, 3, 4, 6, 7, 9, 11, 14, 17, 25, 27, 29], 6: [1, 5, 8, 9, 11, 13, 18, 20, 22, 26, 29, 30], 7: [1, 2, 5, 9, 10, 12, 20, 24, 25, 26, 27, 30], 8: [1, 6, 9, 14, 15, 17, 18, 20, 21, 22, 23, 30], 9: [2, 5, 6, 7, 8, 12, 14, 15, 19, 24, 27, 28], 10: [0, 1, 2, 7, 12, 13, 15, 18, 19, 21, 24, 28], 11: [3, 4, 5, 6, 12, 15, 17, 20, 22, 23, 29, 30], 12: [0, 4, 7, 9, 10, 11, 14, 16, 18, 21, 27, 30], 13: [0, 3, 4, 6, 10, 16, 20, 23, 24, 25, 26, 27], 14: [1, 5, 8, 9, 12, 15, 17, 18, 19, 20, 22, 23], 15: [8, 9, 10, 11, 14, 17, 19, 20, 21, 27, 28, 30], 16: [3, 12, 13, 18, 19, 21, 22, 23, 24, 26, 28, 29], 17: [0, 1, 3, 5, 8, 11, 14, 15, 20, 22, 23, 29], 18: [0, 2, 3, 4, 6, 8, 10, 12, 14, 16, 24, 25], 19: [3, 9, 10, 14, 15, 16, 20, 21, 22, 23, 26, 28], 20: [1, 6, 7, 8, 11, 13, 14, 15, 17, 19, 24, 25], 21: [0, 2, 4, 8, 10, 12, 15, 16, 19, 25, 27, 29], 22: [0, 2, 6, 8, 11, 14, 16, 17, 19, 23, 28, 30], 23: [1, 2, 8, 11, 13, 14, 16, 17, 19, 22, 26, 28], 24: [0, 2, 3, 7, 9, 10, 13, 16, 18, 20, 28, 30], 25: [1, 2, 3, 4, 5, 7, 13, 18, 20, 21, 26, 29], 26: [0, 4, 6, 7, 13, 16, 19, 23, 25, 27, 29, 30], 27: [0, 1, 2, 5, 7, 9, 12, 13, 15, 21, 26, 30], 28: [2, 4, 9, 10, 15, 16, 19, 22, 23, 24, 29, 30], 29: [0, 3, 4, 5, 6, 11, 16, 17, 21, 25, 26, 28], 30: [1, 6, 7, 8, 11, 12, 15, 22, 24, 26, 27, 28]\} }

\vspace{0.5\baselineskip}
\noindent $\bold c = {}$[$1$, $91$, $-39$, $14$, $39$, $33$, $75$, $-48$, $-37$, $2$, $146$, $-14$, $-13$, $23$, $20$, $6$, $-84$, $-32$, $27$, $38$, $-93$, $-66$, $-43$, $21$, $-79$, $-43$, $18$, $-15$, $59$, $1$, $-8$]

\paragraph{Order $\boldsymbol{n = 33}$.} ~\\
{\small \{0: [1, 2, 3, 4, 5, 6, 27, 28, 29, 30, 31, 32], 1: [0, 2, 3, 4, 5, 6, 7, 11, 28, 29, 31, 32], 2: [0, 1, 3, 4, 5, 6, 7, 8, 29, 30, 31, 32], 3: [0, 1, 2, 4, 5, 6, 7, 8, 9, 30, 31, 32], 4: [0, 1, 2, 3, 5, 6, 7, 8, 9, 10, 31, 32], 5: [0, 1, 2, 3, 4, 6, 7, 8, 9, 10, 11, 32], 6: [0, 1, 2, 3, 4, 5, 7, 8, 9, 10, 11, 12], 7: [1, 2, 3, 4, 5, 6, 8, 9, 10, 11, 12, 13], 8: [2, 3, 4, 5, 6, 7, 9, 10, 11, 12, 13, 14], 9: [3, 4, 5, 6, 7, 8, 10, 11, 12, 13, 14, 15], 10: [4, 5, 6, 7, 8, 9, 12, 13, 14, 15, 16, 30], 11: [1, 5, 6, 7, 8, 9, 12, 13, 14, 15, 16, 17], 12: [6, 7, 8, 9, 10, 11, 13, 14, 15, 16, 17, 18], 13: [7, 8, 9, 10, 11, 12, 14, 15, 16, 17, 18, 19], 14: [8, 9, 10, 11, 12, 13, 15, 16, 17, 18, 19, 20], 15: [9, 10, 11, 12, 13, 14, 16, 17, 18, 19, 20, 21], 16: [10, 11, 12, 13, 14, 15, 17, 18, 19, 20, 21, 22], 17: [11, 12, 13, 14, 15, 16, 18, 19, 20, 21, 22, 23], 18: [12, 13, 14, 15, 16, 17, 19, 20, 21, 22, 23, 24], 19: [13, 14, 15, 16, 17, 18, 20, 21, 22, 23, 24, 25], 20: [14, 15, 16, 17, 18, 19, 21, 22, 23, 24, 25, 26], 21: [15, 16, 17, 18, 19, 20, 22, 23, 24, 25, 26, 27], 22: [16, 17, 18, 19, 20, 21, 23, 24, 25, 26, 27, 28], 23: [17, 18, 19, 20, 21, 22, 24, 25, 26, 27, 28, 29], 24: [18, 19, 20, 21, 22, 23, 25, 26, 27, 28, 29, 30], 25: [19, 20, 21, 22, 23, 24, 26, 27, 28, 29, 30, 31], 26: [20, 21, 22, 23, 24, 25, 27, 28, 29, 30, 31, 32], 27: [0, 21, 22, 23, 24, 25, 26, 28, 29, 30, 31, 32], 28: [0, 1, 22, 23, 24, 25, 26, 27, 29, 30, 31, 32], 29: [0, 1, 2, 23, 24, 25, 26, 27, 28, 30, 31, 32], 30: [0, 2, 3, 10, 24, 25, 26, 27, 28, 29, 31, 32], 31: [0, 1, 2, 3, 4, 25, 26, 27, 28, 29, 30, 32], 32: [0, 1, 2, 3, 4, 5, 26, 27, 28, 29, 30, 31]\} }

\vspace{0.5\baselineskip}
\noindent $\bold c = {}$[$1$, $-2$, $1$, $1$, $-2$, $1$, $1$, $-2$, $1$, $1$, $-2$, $1$, $1$, $-2$, $1$, $1$, $-2$, $1$, $1$, $-2$, $1$, $1$, $-2$, $1$, $1$, $-2$, $1$, $1$, $-2$, $1$, $1$, $-2$, $1$]

\paragraph{Order $\boldsymbol{n = 35}$.} ~\\
{\small \{0: [1, 2, 3, 4, 5, 6, 29, 30, 31, 32, 33, 34], 1: [0, 2, 3, 4, 5, 6, 7, 30, 31, 32, 33, 34], 2: [0, 1, 3, 4, 5, 6, 7, 8, 15, 31, 32, 33], 3: [0, 1, 2, 4, 5, 6, 8, 9, 15, 32, 33, 34], 4: [0, 1, 2, 3, 5, 6, 7, 8, 9, 31, 33, 34], 5: [0, 1, 2, 3, 4, 6, 7, 8, 9, 10, 11, 34], 6: [0, 1, 2, 3, 4, 5, 7, 8, 9, 10, 11, 12], 7: [1, 2, 4, 5, 6, 8, 9, 10, 11, 12, 13, 21], 8: [2, 3, 4, 5, 6, 7, 9, 10, 11, 12, 13, 14], 9: [3, 4, 5, 6, 7, 8, 10, 11, 12, 13, 14, 15], 10: [5, 6, 7, 8, 9, 11, 12, 13, 14, 15, 16, 25], 11: [5, 6, 7, 8, 9, 10, 12, 13, 14, 15, 16, 17], 12: [6, 7, 8, 9, 10, 11, 13, 14, 15, 16, 17, 18], 13: [7, 8, 9, 10, 11, 12, 14, 16, 17, 18, 19, 34], 14: [8, 9, 10, 11, 12, 13, 15, 16, 17, 18, 19, 20], 15: [2, 3, 9, 10, 11, 12, 14, 16, 17, 18, 19, 20], 16: [10, 11, 12, 13, 14, 15, 17, 18, 19, 20, 21, 22], 17: [11, 12, 13, 14, 15, 16, 18, 19, 20, 21, 22, 23], 18: [12, 13, 14, 15, 16, 17, 19, 20, 21, 22, 23, 24], 19: [13, 14, 15, 16, 17, 18, 20, 21, 22, 23, 24, 25], 20: [14, 15, 16, 17, 18, 19, 21, 22, 23, 24, 25, 26], 21: [7, 16, 17, 18, 19, 20, 22, 23, 24, 25, 26, 27], 22: [16, 17, 18, 19, 20, 21, 23, 24, 25, 26, 27, 28], 23: [17, 18, 19, 20, 21, 22, 24, 25, 26, 27, 28, 29], 24: [18, 19, 20, 21, 22, 23, 25, 26, 27, 28, 29, 30], 25: [10, 19, 20, 21, 22, 23, 24, 26, 27, 28, 29, 30], 26: [20, 21, 22, 23, 24, 25, 27, 28, 29, 30, 31, 32], 27: [21, 22, 23, 24, 25, 26, 28, 29, 30, 31, 32, 33], 28: [22, 23, 24, 25, 26, 27, 29, 30, 31, 32, 33, 34], 29: [0, 23, 24, 25, 26, 27, 28, 30, 31, 32, 33, 34], 30: [0, 1, 24, 25, 26, 27, 28, 29, 31, 32, 33, 34], 31: [0, 1, 2, 4, 26, 27, 28, 29, 30, 32, 33, 34], 32: [0, 1, 2, 3, 26, 27, 28, 29, 30, 31, 33, 34], 33: [0, 1, 2, 3, 4, 27, 28, 29, 30, 31, 32, 34], 34: [0, 1, 3, 4, 5, 13, 28, 29, 30, 31, 32, 33]\} }

\vspace{0.5\baselineskip}
\noindent $\bold c = {}$[$1$, $-1$, $-1$, $-3$, $3$, $2$, $-1$, $-1$, $1$, $1$, $-2$, $2$, $-2$, $-1$, $3$, $-1$, $-1$, $2$, $-2$, $-2$, $5$, $-1$, $-1$, $1$, $-2$, $-2$, $6$, $-3$, $-1$, $1$, $-1$, $5$, $-1$, $-4$, $1$]

\paragraph{Order $\boldsymbol{n = 37}$.} ~\\
{\small \{0: [1, 2, 3, 4, 5, 6, 31, 32, 33, 34, 35, 36], 1: [0, 2, 3, 4, 5, 6, 7, 18, 22, 32, 33, 35], 2: [0, 1, 3, 4, 5, 6, 7, 8, 33, 34, 35, 36], 3: [0, 1, 2, 4, 5, 6, 7, 8, 9, 34, 35, 36], 4: [0, 1, 2, 3, 5, 6, 7, 8, 9, 10, 35, 36], 5: [0, 1, 2, 3, 4, 6, 7, 8, 9, 10, 11, 36], 6: [0, 1, 2, 3, 4, 5, 7, 8, 9, 10, 11, 12], 7: [1, 2, 3, 4, 5, 6, 8, 9, 10, 11, 12, 13], 8: [2, 3, 4, 5, 6, 7, 9, 10, 11, 12, 13, 14], 9: [3, 4, 5, 6, 7, 8, 10, 11, 12, 13, 15, 32], 10: [4, 5, 6, 7, 8, 9, 11, 12, 13, 14, 15, 16], 11: [5, 6, 7, 8, 9, 10, 12, 13, 14, 15, 16, 17], 12: [6, 7, 8, 9, 10, 11, 13, 14, 15, 16, 17, 18], 13: [7, 8, 9, 10, 11, 12, 14, 15, 16, 17, 19, 36], 14: [8, 10, 11, 12, 13, 15, 16, 17, 18, 19, 20, 35], 15: [9, 10, 11, 12, 13, 14, 16, 17, 18, 19, 20, 21], 16: [10, 11, 12, 13, 14, 15, 17, 18, 19, 20, 21, 22], 17: [11, 12, 13, 14, 15, 16, 18, 19, 20, 21, 22, 23], 18: [1, 12, 14, 15, 16, 17, 19, 20, 21, 22, 23, 24], 19: [13, 14, 15, 16, 17, 18, 20, 21, 22, 23, 24, 25], 20: [14, 15, 16, 17, 18, 19, 21, 22, 23, 24, 25, 26], 21: [15, 16, 17, 18, 19, 20, 22, 23, 24, 25, 26, 27], 22: [1, 16, 17, 18, 19, 20, 21, 23, 24, 26, 27, 28], 23: [17, 18, 19, 20, 21, 22, 24, 25, 26, 27, 28, 29], 24: [18, 19, 20, 21, 22, 23, 25, 26, 27, 28, 29, 30], 25: [19, 20, 21, 23, 24, 26, 27, 28, 29, 30, 31, 34], 26: [20, 21, 22, 23, 24, 25, 27, 28, 29, 30, 31, 32], 27: [21, 22, 23, 24, 25, 26, 28, 29, 30, 31, 32, 33], 28: [22, 23, 24, 25, 26, 27, 29, 30, 31, 32, 33, 34], 29: [23, 24, 25, 26, 27, 28, 30, 31, 32, 33, 34, 35], 30: [24, 25, 26, 27, 28, 29, 31, 32, 33, 34, 35, 36], 31: [0, 25, 26, 27, 28, 29, 30, 32, 33, 34, 35, 36], 32: [0, 1, 9, 26, 27, 28, 29, 30, 31, 33, 34, 36], 33: [0, 1, 2, 27, 28, 29, 30, 31, 32, 34, 35, 36], 34: [0, 2, 3, 25, 28, 29, 30, 31, 32, 33, 35, 36], 35: [0, 1, 2, 3, 4, 14, 29, 30, 31, 33, 34, 36], 36: [0, 2, 3, 4, 5, 13, 30, 31, 32, 33, 34, 35]\} }

\vspace{0.5\baselineskip}
\noindent $\bold c = {}$[$2$, $-3$, $-4$, $5$, $1$, $-1$, $-1$, $-4$, $5$, $2$, $-5$, $1$, $1$, $-1$, $6$, $-5$, $-4$, $7$, $-1$, $-5$, $4$, $-5$, $3$, $6$, $-5$, $-5$, $8$, $-3$, $1$, $1$, $-4$, $3$, $4$, $-7$, $-1$, $3$, $1$]

\paragraph{Order $\boldsymbol{n = 39}$.} ~\\
{\small \{0: [1, 2, 3, 4, 5, 6, 15, 33, 34, 36, 37, 38], 1: [0, 2, 3, 4, 5, 6, 7, 34, 35, 36, 37, 38], 2: [0, 1, 3, 4, 5, 6, 7, 8, 35, 36, 37, 38], 3: [0, 1, 2, 4, 5, 6, 7, 8, 9, 36, 37, 38], 4: [0, 1, 2, 3, 5, 6, 7, 8, 9, 10, 37, 38], 5: [0, 1, 2, 3, 4, 6, 7, 8, 9, 10, 11, 38], 6: [0, 1, 2, 3, 4, 5, 7, 8, 9, 10, 11, 12], 7: [1, 2, 3, 4, 5, 6, 8, 9, 10, 11, 12, 13], 8: [2, 3, 4, 5, 6, 7, 9, 10, 11, 12, 13, 14], 9: [3, 4, 5, 6, 7, 8, 10, 11, 12, 13, 14, 15], 10: [4, 5, 6, 7, 8, 9, 11, 12, 13, 14, 15, 16], 11: [5, 6, 7, 8, 9, 10, 12, 13, 14, 16, 17, 35], 12: [6, 7, 8, 9, 10, 11, 13, 14, 15, 16, 17, 18], 13: [7, 8, 9, 10, 11, 12, 14, 15, 16, 17, 18, 19], 14: [8, 9, 10, 11, 12, 13, 15, 16, 17, 18, 19, 20], 15: [0, 9, 10, 12, 13, 14, 16, 17, 18, 19, 20, 21], 16: [10, 11, 12, 13, 14, 15, 17, 18, 19, 20, 21, 22], 17: [11, 12, 13, 14, 15, 16, 18, 19, 20, 21, 22, 23], 18: [12, 13, 14, 15, 16, 17, 19, 20, 21, 22, 23, 24], 19: [13, 14, 15, 16, 17, 18, 20, 21, 22, 23, 24, 25], 20: [14, 15, 16, 17, 18, 19, 21, 22, 23, 24, 25, 26], 21: [15, 16, 17, 18, 19, 20, 22, 23, 24, 25, 26, 27], 22: [16, 17, 18, 19, 20, 21, 23, 24, 25, 26, 27, 28], 23: [17, 18, 19, 20, 21, 22, 24, 25, 26, 27, 28, 29], 24: [18, 19, 20, 21, 22, 23, 25, 26, 27, 28, 29, 30], 25: [19, 20, 21, 22, 23, 24, 26, 27, 28, 29, 30, 31], 26: [20, 21, 22, 23, 24, 25, 27, 28, 29, 30, 31, 32], 27: [21, 22, 23, 24, 25, 26, 28, 29, 30, 31, 32, 33], 28: [22, 23, 24, 25, 26, 27, 29, 30, 31, 32, 33, 34], 29: [23, 24, 25, 26, 27, 28, 30, 31, 32, 33, 34, 35], 30: [24, 25, 26, 27, 28, 29, 31, 32, 33, 34, 35, 36], 31: [25, 26, 27, 28, 29, 30, 32, 33, 34, 35, 36, 37], 32: [26, 27, 28, 29, 30, 31, 33, 34, 35, 36, 37, 38], 33: [0, 27, 28, 29, 30, 31, 32, 34, 35, 36, 37, 38], 34: [0, 1, 28, 29, 30, 31, 32, 33, 35, 36, 37, 38], 35: [1, 2, 11, 29, 30, 31, 32, 33, 34, 36, 37, 38], 36: [0, 1, 2, 3, 30, 31, 32, 33, 34, 35, 37, 38], 37: [0, 1, 2, 3, 4, 31, 32, 33, 34, 35, 36, 38], 38: [0, 1, 2, 3, 4, 5, 32, 33, 34, 35, 36, 37]\} }

\vspace{0.5\baselineskip}
\noindent $\bold c = {}$[$1$, $-2$, $1$, $1$, $-2$, $1$, $1$, $-2$, $1$, $1$, $-2$, $1$, $1$, $-2$, $1$, $1$, $-2$, $1$, $1$, $-2$, $1$, $1$, $-2$, $1$, $1$, $-2$, $1$, $1$, $-2$, $1$, $1$, $-2$, $1$, $1$, $-2$, $1$, $1$, $-2$, $1$]

\end{appendices}

\end{document}